\documentclass[11pt]{article}
\usepackage{amssymb}
\usepackage{amsthm}
\usepackage{epsfig}
\newtheorem{theorem}{Theorem}
\newtheorem{lemma}[theorem]{Lemma}
\newtheorem{corollary}[theorem]{Corollary}
\newcommand{\N}{\mathbb N}
\newcommand{\Q}{\mathbb Q}

\begin{document}
\title{A combinatorial proof of the Removal Lemma for Groups}
\author{Daniel Kr{\'a}l'\thanks{%
        Institute for Theoretical Computer Science (ITI),
        Faculty of Mathematics and Physics, Charles University,
        Malostransk\'e n\'am\v{e}st\'{\i}~25, 118~00 Prague, Czech
        Republic. E-mail: {\tt kral@kam.mff.cuni.cz}. Institute for
        Theoretical computer science is supported as project 1M0545
        by Czech Ministry of Education.}
        \and
        Oriol Serra \thanks{%
        Departament de Matem\`atica Aplicada IV,
        Universitat Polit\`ecnica de Catalunya. E-mail: {\tt oserra@ma4.upc.edu}.   Supported by the Catalan Research Council
        under project 2005SGR0258.}
        \and
    Llu\'{i}s Vena \thanks{%
        Departament de Matem\`atica Aplicada IV,
        Universitat Polit\`ecnica de Catalunya. E-mail: {\tt lvena@ma4.upc.edu}. Supported by the
        Spanish Research Council under project MTM2005-08990-C01-C02.}
        }
\date{}

\maketitle

\begin{abstract}
Green~[Geometric and Functional Analysis 15 (2005), 340--376]
established a version of the Szemer\'edi Regularity Lemma for
abelian groups and derived the Removal Lemma for abelian groups as
its corollary. We provide another proof of his Removal Lemma that
allows us to extend its statement to all finite groups. We also
discuss possible extensions of the Removal Lemma to systems of equations.
\end{abstract}

\section{Introduction}

A consequence of the celebrated Szemer\'edi Regularity Lemma for
graphs is the so-called Removal Lemma. In its simplest formulation,
the lemma states that a graph with $o(n^3)$ triangles can be
transformed into a triangle-free graph by removing only $o(n^2)$
edges. Green~\cite{green05} has recently established a version of
the Szemer\'edi Regularity Lemma for abelian groups and derived, as
a consequence of his result, the Removal Lemma for abelian groups:

\begin{theorem}[Green~\cite{green05}, Theorem 1.5]
\label{thm-green} Let $G$ be a finite abelian group of order $N$.
Let $m\geq 3$ be an integer, and suppose that $A_1, \ldots, A_m$ are
subsets of $G$ such that there are $o(N^{m-1})$ solutions to the
equation $a_1+\ldots+a_m=0$ with $a_i\in A_i$ for all $i$. Then,
it is possible to remove $o(N)$ elements from each set $A_i$ so as
to obtain sets $A'_i$ such that there is no solution of the equation
$a_1'+\ldots+a_m'=0$ with $a_i'\in A_i'$ for all $i$.
\end{theorem}

Rigorously speaking,
Theorem~\ref{thm-green} asserts that for every $\delta>0$ there
exists $\delta'(\delta,m)$ such that, if the equation
$a_1+\ldots+a_m=0$ has less than $\delta N^{m-1}$ solutions with
$a_i\in A_i$, then there are subsets $A'_i\subseteq A_i$,
$|A_i\setminus A'_i|\le\delta' N$ such that the equation
$a_1+\ldots+a_m=0$ has no solution with $a_i\in A'_i$, and the value
of $\delta'$ tends to $0$ as $\delta\to 0$. Let us emphasize that
the value of $\delta'$ does not depend on the order $N$ of the group
$G$ (nor its structure).

The proof of the Regularity Lemma for abelian groups
in~\cite{green05} relies heavily on Fourier analysis techniques, and
thus the result is restricted only to abelian groups. In this paper
we provide a proof of Theorem~\ref{thm-green} building on
combinatorial methods and thus we are able to generalize the result
to arbitrary finite groups. In particular, we prove the following
extension of Theorem~\ref{thm-green}.

\begin{theorem}
\label{thm-main} Let $G$ be a finite group of order $N$. Let
$A_1,\ldots ,A_m$, $m\ge 2$, be sets of elements of $G$ and let $g$
be an arbitrary element of $G$. If the equation $x_1x_2\cdots x_m=g$
has $o(N^{m-1})$ solutions with $x_i\in A_i$, then there are subsets
$A'_i\subseteq A_i$ with $|A_i\setminus A'_i|=o(N)$ such that there
is no solution of the equation $x_1x_2\cdots x_m=g$ with $x_i\in
A'_i$.
\end{theorem}

We use the multiplicative notation in Theorem~\ref{thm-main} to
emphasize that the group $G$ need not be abelian.

Our technique also allows to extend Theorem~\ref{thm-green} to
equation systems of a certain type. Let $G$ be an abelian group (with
additive notation) and consider an equation system of the following type:
\begin{equation}\label{eq:ab-sys}
\left.\begin{array}{ccccccl}
\epsilon_{11}x_1  & + & \cdots & + & \epsilon_{1m}x_m  & = & 0\\
\vdots &   & &   & \vdots & & \vdots\\
\epsilon_{k1}x_1  & + & \cdots & + & \epsilon_{km}x_m  & = & 0
\end{array}\right\}
\end{equation}
where $\epsilon_{ij}\in \{ -1,0,1\}$, $k\ge 1$ and $m\ge 2$.
The vector $(\epsilon_{i1},\ldots ,\epsilon_{im})$ is referred to as the
{\em characteristic vector} of the $i$--th equation. We say that the
system (\ref{eq:ab-sys}) is  {\em graph representable} by a directed
graph $H$ with $m$ arcs (one for each variable in the system) if the
characteristic vectors of cycles in $H$ are precisely integer linear
combinations of the characteristic vectors of the equations, see
Section~\ref{sec:sys} for details. With this notation, our second
result is the following:

\begin{theorem}
\label{thm-mainsys} Let $G$ be a finite abelian  group of order $N$.
Let $A_1,\ldots,A_m$, $m\ge 2$, be sets of elements of $G$. If the
equation system (\ref{eq:ab-sys})   is graph-representable and has
$o(N^{m-k})$ solutions with $x_i\in A_i$, then there are subsets
$A'_i\subseteq A_i$ with $|A_i\setminus A'_i|=o(N)$ such that there
is no solution of the system (\ref{eq:ab-sys}) with $x_i\in A'_i$.
\end{theorem}

Theorem \ref{thm-mainsys} can be also extended to non-abelian groups
at the expense of strengthening the notion of graph
representability. With this strong version which is explained in
Section~\ref{sec:sys} the same technique allows us to prove:

\begin{theorem}
\label{thm-mainsys-nonab} Let $G$ be a finite    group of order $N$
written multiplicatively. Let $A_1,\ldots,A_m$, $m\ge 2$, be sets of
elements of $G$. Consider the equation system
\begin{equation}\label{eq:sysnonab}
\left.\begin{array}{ccccl}
x_{\sigma_1(1)}^{\epsilon_{1\sigma_1(1)}} & \cdots & x_{\sigma_1(m)}^{\epsilon_{1\sigma_1(m)}} & = & 1\\
\vdots & & \vdots & & \vdots\\
x_{\sigma_k(1)}^{\epsilon_{k\sigma_k(1)}} & \cdots &
x_{\sigma_k(m)}^{\epsilon_{k\sigma_k(m)}} & = & 1
\end{array}\right\}
\end{equation}
where $\sigma_1, \ldots ,\sigma_k$ are permutations of $[1,m]$,
$\epsilon_{ij}\in \{ -1,0,1\}$, $k\ge 1$, $m\ge 2$. If the system is
\emph{strongly graph-representable} and has $o(N^{m-k})$ solutions
with $x_i\in A_i$, then there are subsets $A'_i\subseteq A_i$ with
$|A_i\setminus A'_i|=o(N)$ such that there is no solution of the
system (\ref{eq:ab-sys}) with $x_i\in A'_i$.
\end{theorem}

Before proceeding further,
let us present a corollary of Theorem~\ref{thm-mainsys-nonab}
which illustrates possible applications of our results.
Let $G$ be a finite group and $A, B\subseteq G$. The
representation function $r_{A,B}:G\rightarrow \N$, defined as
$r_{A,B}(g)=\left|\{ (a,b)\in A\times B: ab=g\}\right|$, counts the
number of representations of an element $g\in G$ as  a product of an
element in $A$ and one in $B$. We write $r_A$ for $r_{A,A}$.

\begin{corollary}
\label{cor:appl} Let $G$ be a finite group of order $N$ and let $A,
B, C, D, E\subseteq G$. If
$$\frac{1}{N}\sum_{g\in E} r_{A,B}(g)r_{C,D}(g)=o(N^2)\;\mbox{,}$$
then it is possible to eliminate $o(N)$ elements in each of the sets to obtain
sets $A', B', C', D', E'$ such that
$$\sum_{g\in E'} r_{A',B'}(g)r_{C',D'}(g)=0\;\mbox{.}$$
In particular,
\begin{enumerate}
\item If $\frac{1}{N}\sum_{g\in E} r_A^2(g)=o(N^2)$, then $(A')^2\cap E'=\emptyset$ ($A'$ is $E'$--product--free).
\item If $\frac{1}{N}\sum_{g\in G\setminus A} r_A^2(g)=o(N^2)$, then $|(A')^2|=|A'|+o(N)$ ($A'$ has small doubling).
\item If $\frac{1}{N}\sum_{g\in G} r_{A,B}(g)r_{B,A}(g)=o(N^2)$, then $|A'B'\cap B'A'|=o(N)$ (almost all pairs do not commute).
\end{enumerate}
\end{corollary}

\begin{proof}
Consider the following equation system:
\begin{equation}\label{eq:sys2}
\left.\begin{array}{lcl}
x_1x_2x_4^{-1}x_3^{-1} & = & 1\\
x_1x_2x_5^{-1} & = & 1
\end{array}\right\}
\end{equation}
The system (\ref{eq:sys2}) is strongly representable by the graph $H$
depicted in Figure~\ref{figappl}.

%\begin{figure}
%\begin{center}
%\epsfbox{grremove.1}
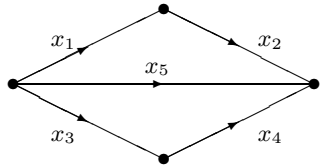
\begin{figure}[ht] \hspace{2cm}
 \setlength{\unitlength}{5mm}
 \begin{center}
 \begin{picture}(16,5){\small
\put(4,3){\vector(2,-1){2}}
 \put(8,1){\vector(2,1){2}}
\put(8,5){\vector(2,-1){2}}
 \put(4,3){\vector(2,1){2}}
\put(4,3){\vector(1,0){4}}
\put(4,3){\line(2,-1){4}}%
\put(5,1.5){{\footnotesize $x_3$}}%
 \put(8,1){\line(2,1){4}}%
 \put(10.5,1.5){{\footnotesize $x_4$}}%
\put(12,3){\line(-2,1){4}}%
\put(10.5,4){{\footnotesize $x_2$}}%
 \put(4,3){\line(2,1){4}}%
\put(5,4){{\footnotesize $x_1$}}
  \put(4,3){\line(1,0){8}}%
\put(7.5,3.3){{\footnotesize $x_5$}} \put( 4,3){\circle*{0.3}}
 \put(8,1){\circle*{0.3}}
\put(8,5){\circle*{0.3}}
 \put(12, 3){\circle*{0.3}}
}
\end{picture}
\end{center}
\caption{A directed graph representing the equation system
(\ref{eq:sys2}).} \label{figappl}
\end{figure}

The number of solutions of (\ref{eq:sys2}) with $x_1\in A$, $x_2\in
B$, $x_3\in C$, $x_4\in D$ and $x_5\in E$ is $\sum_{g\in E}
r_{A,B}(g)r_{C,D}(g)$. Hence, if it holds that
$$\frac{1}{N}\sum_{g\in E} r_{A,B}(g)r_{C,D}(g)=o(N^2)\;\mbox{,}$$
then there are $o(N^3)$ solutions of the system (\ref{eq:sys2}). By
Theorem~\ref{thm-mainsys-nonab} applied with $m=5$, $k=2$, $A_1=A,
A_2=B, A_3=C, A_4=D$ and $A_5=E$, it is possible to remove $o(N)$
elements from each of the sets $A,\ldots,E$ obtaining sets
$A',\ldots,E'$ such that the system (\ref{eq:sys2}) has no solution
with $x_1\in A',\ldots,x_5\in E'$.

Applying the above argument with $A=B=C=D$, we obtain that
$\sum_{g\in E'} r_{A'}^2(g)=0$, which is equivalent to $(A')^2\cap
E'=\emptyset$. This proves 1. Setting $E=G\setminus A$, we get
$\sum_{g\in E'} r_{A'}^2(g)=0$. Since $(A')^2\subseteq A\cup
(E\setminus E')$, $|A\setminus A'|=o(N)$, $|E\setminus E'|=o(N)$, we
obtain 2. Similarly, 3 is derived by applying the Corollary for
$A=C$ and $B=D$.
\end{proof}

\section{Removal Lemma for groups} \label{sec:main}

In our arguments, the following consequence of a variant of Szemer\'edi Regu\-larity Lemma
for directed graphs becomes useful:

\begin{lemma}[Alon and Shapira~\cite{alon04}, Lemma 4.1]
\label{lm-alon} Let $H$ be a fixed directed graph of order $h$. If
$G$ contains less than $o(n^h)$ copies of $H$, there exists a set
$E$ of at most $o(n^2)$ arcs of $G$ such that the graph obtained
from $G$ by removing the arcs of $E$ is $H$-free.
\end{lemma}

The proofs of Theorems~\ref{thm-main}, \ref{thm-mainsys} and
\ref{thm-mainsys-nonab} consist in constructing a blow-up graph of
a small graph
$H$ (which is a cycle in the case of Theorem \ref{thm-main}) such that any
solution of the equations gives rise to $N$ copies of $H$ and every
copy of $H$ comes in fact from a solution of the equations. We then
apply the removal lemma for graphs and, by a pigeonhole principle,
reduce the $o(N^2)$ arcs from Lemma \ref{lm-alon} to the $o(N)$
elements stated in Theorem \ref{thm-main}.

\begin{proof}[Proof of Theorem~\ref{thm-main}]
Fix $\delta_0>0$ and $m\ge 2$. Let $G$ be a finite group of order
$N$, let $g$ be an element of $G$ and let $A_1,\ldots,A_m$ be sets
of elements of $G$.

We define an auxiliary directed graph $H_0$ whose vertex set is the
set $G\times \{ 1,\ldots ,m\}$, i.e., they are pairs formed by an
element of the group $G$ and an integer between $1$ and $m$. There
is an arc in $H_0$  from a vertex $(x,i)$, $1\le i\le m-1$,  to a
vertex $(y,i+1)$ if there exists an element $a_i\in A_i$ such that
$xa_i=y$. This arc is labeled by the pair $[a_i,i]$. $H_0$ also
contains an arc from a vertex $(x,m)$ to a vertex $(y,1)$ if there
exists an element $a_m\in A_m$ such that $xa_mg^{-1}=y$. This arc is
labeled by the pair $[a_m,m]$. Let $N_0=mN$ denote the order of
$H_0$. Note that, for each element $a_i\in A_i$, $H_0$ contains
exactly $N$ arcs labelled with $[a_i,i]$.

Observe that any directed cycle of $H_0$ with length $m$  gives a
solution of the equation: if $[a_1,1], [a_2,2],\ldots ,[a_m,m]$ are
the labels of the arcs in the cycle and it contains the vertex
$(z,1)$, then $za_1a_2\cdots a_mg^{-1}=z$ by the definition of
$H_0$. In the opposite way, each solution $a_1,\ldots,a_m$ of
(\ref{eq-1}) corresponds to $N$ edge disjoint directed cycles of
length $m$ in $H_0$:
\begin{equation}
(z,1),(za_1,2),(za_1a_2,3),\ldots,(za_1\ldots a_{m-1},m),(za_1\ldots
a_mg^{-1},1)=(z,1)\label{eq-2}
\end{equation}
one for each of the $N$ distinct possible choices of $z\in G$.

Suppose, using the hypothesis, that there are less than $\delta_0
N^{m-1}$ solutions of the equation
\begin{equation}
x_1x_2\cdots x_m=g \mbox{ with $x_i\in A_i$.}\label{eq-1}
\end{equation}
By the correspondence of the cycles of $H_0$ and the solutions of
(\ref{eq-1}), the directed graph $H_0$ contains no more than
$\delta_0 N^m$ distinct directed cycles of length $m$.

Apply Lemma~\ref{lm-alon} to $H_0$ and the directed cycle of length
$m$ with $\delta=\delta_0/m^m$: since $H_0$ has less than
$\delta_0N^m=\delta N_0^m$ copies of the directed cycle of length
$m$, there is a set $E$ of at most $\delta'N_0^2$ arcs such that
$H_0-E$ contains no directed cycle of length $m$ with some $\delta'$
depending only on $\delta$ and $m$.

Let $B_i$ be the set of those elements $a\in A_i$ such that $E$
contains at least $N/m$ arcs labeled with $[a,i]$. Since
$|E|\le\delta' N_0^2$, the size of each $B_i$ is at most
$m|E|/N\le\delta' m^3 N$. Set $A'_i=A_i\setminus B_i$. Since the
size of $B_i$ is bounded by $\delta' m^3 N$, $\delta'$ depends on
$\delta$ and $m$ only, and $\delta'\to 0$ as $\delta_0\to 0$, the
theorem will be proven after we show that there is no solution of
the equation (\ref{eq-1}) with $a_i\in A'_i$.

Assume that there is a solution with $a_i\in A'_i$ of the equation
(\ref{eq-1}). Consider the $N$ edge disjoint directed cycles of
length $m$ corresponding to $a_1,\ldots,a_m$ which are given by
(\ref{eq-2}). Each of these $N$ cycles contains at least one of the
arcs of $E$ and the arcs of these $N$ edge disjoint cycles are
labelled only with the pairs $[a_1,1]$, $[a_2,2]$, \dots, $[a_m,m]$.
Since these directed cycles are disjoint, the set $E$ contains at
least $N/m$ arcs labelled $[a_i,i]$ for some $1\le i\le m$.
Consequently, $a_i\in B_i$ and thus $a_i\not\in A'_i$. We conclude
that there is no solution of (\ref{eq-1}) with $a_i\in A'_i$.
\end{proof}

\section{Extensions to systems of equations}\label{sec:sys}

Let us now recall the notion of cycle spaces of directed graphs. If
$H$ is a directed graph with $m$ arcs, then the {\em cycle space} of
$H$ is the vector space over $\Q$ spanned by the characteristic
vectors of cycles of $H$ where the {\em characteristic vector} of
a cycle $C$ of $H$ is the $m$-dimensional vector $v$ with each coordinate
associated with one of the arcs such that the $i$-th coordinate of
$v$ is $+1$ if the $i$-th arc is traversed by $C$ in its direction,
it is $-1$ if it is traversed by $C$ in the reverse direction, and
it is $0$ if the arc is not traversed by $C$.

A set of integer vectors contained in the cycle space
is said to {\em integrally generate} the cycle space of $H$
if they are independent and every vector of the cycle space can be
expressed as a linear combination of these vectors with integer
coefficients. It is known~\cite{liebchen} that the vectors integrally
generate the cycle space if and only if every maximum
square submatrix of the matrix formed by these vectors
has determinant $0$, $+1$ or
$-1$. This turns out to be equivalent to the fact that a determinant
of one such non-singular submatrix is $+1$ or $-1$. Let us now give some
examples. If $T$ is a spanning tree of $H$, then
the characteristic vectors of the fundamental cycles
with respect to $T$ always integrally generate the cycle space
of $H$~\cite{liebchen}. On the other hand,
an example of
a set of characteristic vectors that generate but not integrally generate
the cycle space of a graph is given in Figure~\ref{fig-ex}.

\begin{figure}[h,t]
%\begin{center}
%\epsfbox{grremove.2}
%\end{center}
\setlength{\unitlength}{5mm}
 \begin{center}
 \begin{picture}(16,5)
    \put(8,3){\vector(-2,1){2}}
    \put(8,3){\vector(2,1){2}}
    \put(4,1){\vector(2,1){2}}
    \put(12,1){\vector(-2,1){2}}
    \put(4,5){\vector(0,-1){2}}
     \put(12,5){\vector(0,-1){2}}
    \put(8,3){\line(-2,1){4}}
    \put(8,3){\line(2,1){4}}
    \put(4,1){\line(2,1){4}}
    \put(12,1){\line(-2,1){4}}
    \put(4,5){\line(0,-1){4}}
     \put(12,5){\line(0,-1){4}}
    \put(5.8,4.2){{\footnotesize $e_0$}}%
    \put(3,2.8){{\footnotesize $e_1$}}
    \put(5.8,1.5){{\footnotesize $e_2$}}
    \put(9.5,4.2){{\footnotesize $e_3$}}
    \put(12.5,2.8){{\footnotesize $e_4$}}
    \put(9.5,1.5){{\footnotesize $e_5$}}
    \put(4,1){\circle*{0.3}}
    \put(4,5){\circle*{0.3}}
    \put( 12,1){\circle*{0.3}}
    \put(12,5){\circle*{0.3}}
    \put(8,3){\circle*{0.3}}
\end{picture}
\end{center}
\caption{An example of a set of cycles generating but not integrally
         generating the cycle space of a directed graph:
         the cycles $e_0e_1e_2e_3e_4e_5$ and $e_0e_1e_2e_5^{-1}e_4^{-1}e_3^{-1}$
     generate the cycle space of the depicted directed graph
     but they do not integrally generate it: the cycle $e_0e_1e_2$ can only be written
     as a rational (not integral) linear combination of the two cycles in the generating set.} \label{fig-ex}
\end{figure}
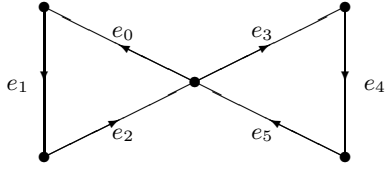

Consider now the   equation system (\ref{eq:ab-sys}). The vector
$(\epsilon_{i1},\ldots ,\epsilon_{im})$ is referred to as the {\em
characteristic vector} of the $i$--th equation. The system is said
to be {\em graph-representable} if there exists a directed graph $H$
with $m$ arcs, each associated with one of the variables
$x_1,\ldots,x_m$, such that the characteristic vectors of the
equations integrally generate the cycle space of $H$. Such a
directed graph $H$ is called a {\em graph representation} of the
equation system (\ref{eq:ab-sys}). Note that the condition that
the characteristic vectors of the equations integrally generate
the cycle space can be efficiently tested since it is equivalent
to computing the value of the determinant of a matrix as explained
in the previous paragraph.

The proof of Theorem~\ref{thm-mainsys} follows the lines of the one
for Theorem \ref{thm-main}. In this case we use the following
colored version of Lemma~\ref{lm-alon}.

\begin{lemma}[Removal Lemma for arc-colored directed graphs] \label{lem-remcol}
Let $m$ be a fixed integer and $H$ a directed graph with its arcs
colored with $m$ colors. If a directed graph $G$ with edges colored
with $m$ colors contains less than $o(n^h)$ copies of $H$ (the
colors of edges in the copy and $H$ must be the same), there exists
a set $E$ of at most $o(n^2)$ arcs such that the graph obtained from
$G$ by removing the arcs contained in $E$ is $H$-free.
\end{lemma}

Lemma \ref{lem-remcol} can be proved by combining the proof of Lemma
\ref{lm-alon} with the edge-colored version of the Regularity Lemma
stated for instance in \cite[Lemma 1.18]{simonovits}.

\begin{proof}[Proof of Theorem~\ref{thm-mainsys}]
Let $H$ be a graph representation of the equation system
(\ref{eq:ab-sys}). We can assume without loss of generality that $H$
is connected. We view the arc corresponding to the variable $x_i$ as
colored with the color $i$. In this way, the arcs of $H$ are colored
with numbers from $1$ to $m$. Since the dimension of the cycle space
of $H$ is $k$ (as the characteristic  vectors of the equations from
(\ref{eq:ab-sys}) are assumed to be independent) and $H$ is
comprised of $m$ arcs, the number of the vertices of $H$ is
$h=m-k+1$.

Next, we construct an auxiliary directed graph $H_0$. The vertex set
of $H_0$ is $G\times V(H)$. For every arc $(u,v)$ of $H$ associated
with $x_i$, the directed graph $H_0$ contains $N|A_i|$ arcs from
$(g,u)$ to $(ga,v)$, one for each $g\in G$ and each $a\in A_i$. The
arc from $(g,u)$ to $(ga,v)$ is colored $i$ and labeled by the pair
$[a,i]$. The order of $H$ is $N_0=hN$, its size is
$N(|A_1|+\cdots+|A_m|)$ and its arcs are colored with numbers
$1,\ldots,m$. We call $H_0$ the \emph{blowup graph} of $H$ by
$A_1,\ldots,A_m$.

Let $H'$ be a subgraph of $H_0$ isomorphic to $H$ (preserving the
colors). The arc of $H'$ colored with $i$ is an arc from a vertex
$(g,u)$ to a vertex $(ga_i,v)$ for some $a_i\in A_i$. Setting
$x_i=a_i$ yields a solution of the system (\ref{eq:ab-sys}): indeed,
if $C$ is a cycle corresponding to the $j$-th equation, then the
cycle $C$ is also present in $H'$ as a cycle $(g_1,u_1)(g_2,u_2)
\ldots(g_l,v_l)$. If $\gamma_t$ is the color of the arc
$\left(\left(g_t,u_t\right), \left(g_{t+1},u_{t+1}\right)\right)$
(indices taken modulo $l$), then $a_{\gamma_t}=g_{t+1}-g_t$, if the
arc is traversed in its direction, and $a_{\gamma_t}=g_{t}-g_{t+1}$
otherwise, and thus
$$0=\sum_{i=1}^l \left(g_{i+1}-g_i\right)=\sum_{i=1}^l \epsilon_{j\gamma_i}a_{\gamma_i}=
     \epsilon_{j1}a_1+\ldots+\epsilon_{jm}a_m\;\mbox{.}$$
Note that we can freely rearrange the summands in the above equation
as the group $G$ is abelian.

We have seen that every subgraph of $H_0$ isomorphic to $H$
corresponds to a solution of the system (\ref{eq:ab-sys}). Let us
now show that every solution of (\ref{eq:ab-sys}) corresponds to $N$
edge disjoint copies of $H$. Fix a vertex $u_0$ of $H$, an element
$z$ of $G$ and a solution of the system $a_1\in A_1,\ldots,a_m\in
A_m$. Define $\varphi: V(H)\to G$ such that $\varphi(u_0)=z$ and
$\varphi(u')-\varphi(u)=a_i$ for an arc $(u,u')$ of $H$
corresponding to the variable $x_i$. By the graph representability
of the system, the function $\varphi$ is well defined:  if there are
two paths from $u_0$ to a vertex $u$ in $H$ they close a cycle $C$
which can be expressed as an integral linear combination of the
cycles in the system. Since the $a_i$'s form a solution of the
system, the sum of the labels on the edges along each of the cycles
arising from the system is zero, and therefore this is also the case
for $C$. Since $H$ is connected, the set of vertices $\{(u,\varphi
(u)), u\in V(H)\}$ induce a copy of $H$ in $H_0$. Since there are
$N$ choices for $z$, and two different choices yield edge--disjoint
copies of $H$, every solution of the system with $a_i\in A_i$ gives
rise to $N$ edge--disjoint copies of $H$.

The proof now proceeds as in Theorem \ref{thm-main} except that
instead of a cycle of length $m$ we aim to consider copies of the
graph $H$. Fix $\delta_0>0$ and apply Lemma~\ref{lem-remcol} for
$\delta=\delta_0/h^h$ which yields $\delta'>0$. If there are less
than $\delta_0 N^{m-k}=\delta_0 N^{h-1}$ solutions of the system
(\ref{eq:ab-sys}), the directed graph $H_0$ contains at most
$\delta_0 N^h=\delta N_0^h$ distinct copies of $H$. By the choice of
$\delta$, there is a set $E$ of at most $\delta' N_0^2$ arcs such
that $H_0\setminus E$ has no copy of $H$.

Let $B_i$ be the set of those elements $a\in A_i$ such that $E$
contains at least $N/m$ arcs $\left((g,u),(ga,v)\right)$ colored
with $i$. Since $|E|\le\delta' N_0^2$, the size of each $B_i$ is at
most $m|E|/N\le\delta' mN_0^2/N=\delta' mh^2 N\le\delta' m^3 N$. Set
$A'_i=A_i\setminus B_i$. Since the size of $B_i$ is bounded by
$\delta' m^3 N$, and $\delta'\to 0$ as $\delta_0\to 0$, the theorem
will be proven after we show that there is no solution of the system
(\ref{eq:ab-sys}) with $a_i'\in A'_i$.

Assume that there is a solution $a_1',\ldots,a_m'$ of the equation
system (\ref{eq:ab-sys}) such that $a_i'\in A'_i$ and consider the
$N$ disjoint copies of $H$ corresponding to this solution. For every
$i$, the $N$ copies of $H$ contain together $N$ arcs colored with
$i$ that are of the form $\left((g,u),(ga'_i,v)\right)$. Hence,
there exists an $i_0$ such that $E$ contains at least $N/m$ arcs
that are colored with $i_0$ and are of the form
$\left((g,u),(ga'_i,v)\right)$. Consequently, $a_i'\in B_i$ and thus
$a_i'\not\in A'_i$ which violates the choice of the solution.
\end{proof}

We have already mentioned that, if the characteristic vectors of the
equations from a system corresponding to fundamental cycles of a
graph $H$ with respect to one of its spanning trees,
then the equation system is graph-representable. If there exists
a representation of this special type,
then we say the system is strongly graph representable.
More precisely, the system (\ref{eq:sysnonab}) is {\it strongly
graph representable} if there is a directed graph $H$ with $m$ arcs
colored by $1,\ldots ,m$ and a spanning tree $T$ of $H$ such that
the fundamental cycles of $H$ with respect to $T$ are cycles
$C_i$ defined as follows: $C_i$ is the cycle traversing the arcs of $H$
in the order $e_{i1}\ldots e_{im}$ where some of $e_{ij}$ are ``empty'',
i.e., they do not define an arc of $H$. If $\epsilon_{i\sigma_i(j)}=+1$,
then $e_{ij}$ is the arc colored with $\sigma_i(j)$ traversed in its direction,
if $\epsilon_{i\sigma_i(j)}=-1$, then $e_{ij}$
is the arc colored with $\sigma_i(j)$
traversed in the opposite direction, and if $\epsilon_{i\sigma_i(j)}=0$,
then $e_{ij}$ is empty.
Note that the condition on the equation system being strongly representable
implies that every equation contains a variable that is not in any of
the other equations.
An example can be found in Figure~\ref{figappl} where the graph strongly
represents the equation system arising from Corollary~\ref{cor:appl}.

This stronger condition suffices to extend
Theorem~\ref{thm-mainsys} to the non-abelian case.

\begin{proof}[Proof of Theorem~\ref{thm-mainsys-nonab}]
The proof is
analogous to the one of Theorem \ref{thm-mainsys}. Let $H$ be a
a strong representation of the system. In particular, $H$ has $m$ edges and
$h=m-k+1$ vertices. Let $H_0$ be the graph with vertex set $G\times V(H)$
that contains
with an arc $((g,u),(ga,v))$ for each arc $(u,v)$ in $H$ that has color $i$ and
each $a\in A_i$.  Such an arc has also color $i$ in $H_0$.

Let $H'$ be a subgraph of $H_0$ that is isomorphic to $H$ (preserving
the colors of the edges). If $((g,u),(g',v))$ is an arc of $H'$
colored with $i$, set $x_i=g^{-1}g'$. Observe that $x_i\in A_i$.
We claim that $a_i$ is a solution of the equation system. Consider
the $i$-th equation and let $C_i=e_{i1}\ldots e_{im}$ be the cycle
corresponding to this equation in $H$ (and thus in $H'$).
Let $g_j$, $j=0,\ldots,m$, be the element of the group $G$
assigned to the vertex shared by the edges $e_{ij}$ and $e_{i,j+1}$.
In particular, $g_0=g_m$ and if $e_{ij}$ is empty, then $g_{j-1}=g_j$.
We infer from the choice of $x_i$ the following:
$$\prod_{j=1}^m x_{\sigma_i(j)}^{\epsilon_{i\sigma_i(j)}}=
  \prod_{j=1}^m g_{j-1}^{-1}g_j=
  g_0^{-1}g_1g_1^{-1}\cdots g_{m-1}g_{m-1}^{-1}g_m=g_0^{-1}g_m=1\;\mbox{.}$$
Hence, $x_i$'s are indeed a solution of the equation system.

On the other hand, each solution $x_i\in A_i$ gives rise
to $N$ edge--disjoint copies of $H$ in $H_0$. Indeed,
let $T$ be a spanning tree of $H$ such that the cycles $C_1,\ldots ,C_m$
corresponding to the equations of the system (\ref{eq:sysnonab})
are fundamental cycles with respect to $T$. Root $T$ at an arbitrarily
chosen vertex $v_0$. Set $g_{v_0}$ to an arbitrary element of $G$ and
define the values $g_v$
for other vertices of the graph $H$ as follows: if $v'$ is the parent
of $v$ in $T$ and the arc $vv'$ has color $i$ and
is oriented from $v$ to $v'$, then $g_v=g_{v'}x_i^{-1}$;
if the arc is oriented from $v'$ to $v$, then $g_v=g_{v'}x_i$.

Let $H'$ be the subgraph of $H_0$ with the vertices $(g_v,v)$
that contains the arc from $(g_v,v)$ to $(g_{v'},v')$ with color $k$
for every arc $vv'$ of $H$ with color $k$.
In order to be sure that $H'$ is properly defined, we have to verify
that an arc from $(g_v,v)$ to $(g_{v'},v')$ with the color $k$
is present in $H_0$. If $vv'$ is an arc of $T$, then $H_0$ contains
the arc from $(g_v,v)$ to $(g_{v'},v')$ by the definition of $g_v$.
If $vv'$ is not contained in $T$, there is exactly one equation
in the system that contains the variable $x_k$. We infer by a simple
manipulation from the definition of $g_v$ that $g_{v'}=g_vx_k$.
Since $x_k\in A_k$, the arc from $(g_v,v)$ to $(g_{v'},v')$
is contained in $H$ and its colors is $k$. Since the choice of $g_{v_0}$
was arbitrary, $H_0$ contains $N$ edge-disjoint copies of $H$.

The rest of the proof is the same as the last three paragraphs of
the proof of Theorem~\ref{thm-mainsys}.
\end{proof}

As a final remark, we briefly discuss the condition of graph
representability. The key point in the  proof of Theorem
\ref{thm-mainsys} is the correspondence between copies of $H$ and
solutions of the system: every copy of $H$ in the constructed graph
$H_0$ yields a solution of the system and every solution gives rise
to $N$ edge-disjoint copies of $H$.
This correspondence can be broken if the system is not graph
representable in the sense we have defined. For instance, in the
example from Figure~\ref{fig-ex}, it is possible to express only $2C$
as an integer combination of the base and thus the stated correspondence
need not exist for groups with elements of order two.

\section*{Acknowledgement}

This work was initiated at Spring School on Combinatorics 2007 held
in Vysok{\'a} L{\'\i}pa in April 2007 which was organized by
DIMATIA, the center for Discrete Mathematics Theoretical Computer
Science and Applications of Charles University in Prague.

\end{document}